\documentclass[a4paper]{amsart}

\usepackage[utf8]{inputenc}
\usepackage[T1]{fontenc}
\usepackage[english]{babel}
\usepackage{babelbib}
\usepackage{enumerate}
\usepackage{url}
\usepackage{amssymb}
\usepackage{dirtytalk}
\usepackage{todonotes} 
\usepackage{microtype}
\usepackage{scalerel}
\usepackage{fullpage}
\usepackage{mathtools}

\DeclareFontFamily{U}{skulls}{}
\DeclareFontShape{U}{skulls}{m}{n}{ <-> skull }{}

\usepackage[hypertexnames=false,backref=page,
 	pdfpagemode=UseNone,
 	breaklinks=true,
 	extension=pdf,
 	colorlinks=true,
 	linkcolor=blue,
 	citecolor=red,
 	urlcolor=blue,
 ]{hyperref} 
\usepackage{cleveref} 
\usepackage{color}

\usepackage{tikz}
\usepackage{tikz-cd}

\usetikzlibrary{matrix}
\usetikzlibrary{arrows}

\usepackage{graphicx}
\usepackage{mathrsfs}

\renewcommand{\mathcal}{\mathscr}

\newcommand{\bbP}{\mathbb{P}}
\newcommand{\bbQ}{\mathbb{Q}}

\newcommand{\sym}{{\bbS}}

\newcommand{\gon}{\textup{gon}}

\renewcommand{\epsilon}{\varepsilon}

\renewcommand{\hom}{\Hom}

\newcommand{\homs}{\textup{\ul{Hom}}}

\renewcommand{\sp}[1]{\cite[\href{https://stacks.math.columbia.edu/tag/#1}{Tag~#1}]{Sp18}}

\makeatletter
\newcommand{\vast}{\bBigg@{4}}
\newcommand{\Vast}{\bBigg@{5}}

\newcommand{\Vastl}{\mathopen\Vast}
\newcommand{\Vastm}{\mathrel\Vast}
\newcommand{\Vastr}{\mathclose\Vast}
\makeatother

\DeclareMathOperator{\spec}{Spec}

\DeclareMathOperator{\Hom}{Hom}

\DeclareMathOperator{\im}{Im}

\renewcommand{\le}{\leqslant}
\renewcommand{\ge}{\geqslant}

\newcommand{\ul}{\underline}
\newcommand{\ol}{\overline}

\renewcommand{\sym}{\textup{Sym}}

\theoremstyle{plain}
\newtheorem{theoremintro}{Theorem}

\newtheorem{theorem}{Theorem}[section]
\newtheorem{lemma}[theorem]{Lemma}

\theoremstyle{definition}

\newtheorem{definition}[theorem]{Definition}
\newtheorem{example}[theorem]{Example}

\newtheorem{remark}[theorem]{Remark}

\numberwithin{equation}{section}

\begin{document}

\title{Points of low degree on curves over function fields}

\begin{abstract}
    We show that the geometric classification of smooth projective curves admitting infinitely many points of degree $d\le 5$ extends from number fields to function fields of characteristic 0. Over number fields, this classification was established by Faltings for $d=1$, Harris--Silverman for $d=2$, Abramovich--Harris for $d=3,4$ and Kadets--Vogt for $d=4,5$. Our approach uses a specialization argument to reduce the problem over function fields to the number field case.
\end{abstract}

\author[S. van Schaick]{Si\`ena van Schaick}
\address{IMAPP Radboud University Nijmegen\\
    PO Box 9010 \\
    6500GL Nijmegen\\
    The Netherlands.}
\email{siena.vanschaick@ru.nl}

\makeatletter
\@namedef{subjclassname@2020}{
    \textup{2020} Mathematics Subject Classification}
\makeatother

\subjclass[2020]{14G99 (11G30, 14H45, 14G05)}
\keywords{rational points, higher degree points, function fields, (algebraic) curves}

\maketitle

\setcounter{tocdepth}{1}

\section{Introduction}

In this paper we adapt results of \cite{HS91}, \cite{AH91}, and \cite{KV25} on the minimal potential density degree of curves over number fields, to curves over function fields of characteristic 0. In particular, we generalize their results on curves over number fields to curves over fields which are finitely generated over $\bbQ$.

We work with the following definitions. A \textit{function field} over a field $k$ is a finitely generated field extension of $k$ which has transcendence degree at least 1. With a \textit{curve} $X$ over a field $k$ we mean a geometrically integral, separated scheme of finite type over $k$ of dimension 1. We say that the points of degree $d$ on $X$ are \textit{potentially dense} if there exists a finite field extension $l/k$ such that the set of points of degree $d$ on $X_l$ is Zariski dense in $X_l$. We define the \textit{potential density degree set} $\wp(X/k)$ to be the set of positive integers $d$ such that the set of points of degree $d$ on $X$ is potentially dense in $X$, and we define the \textit{minimal potential density degree} $\min(\wp(X/k))$ as the minimum of this set.

We first note two distinct cases in which a projective curve $X$ over a number field or function field of characteristic 0 has a (potentially) dense set of points of low degree. Firstly, if $X$ admits a finite morphism of degree $d$ to $\mathbb{P}^1$ or an elliptic curve $E$ of positive rank, the rational points of $\mathbb{P}^1$ or $E$ can be pulled back to a dense set of points of degree at most $d$ on $X$. Secondly, if $X$ is (the normalization of) a DF curve of type $d$ (see \Cref{def:DF_curve}), then $X$ has a potentially dense set of points of degree $d$. The name \say{DF curve} is inspired by the work of Debarre and Fahlaoui, who in \cite{DF93} constructed the first examples of DF curves to show the existence of curves with a dense set of points of degree $d$ \textit{not} coming from a morphism to $\mathbb{P}^1$ or an elliptic curve.

The central problem in this paper is to give a geometric characterization of smooth projective curves over function fields of characteristic 0 with minimal potential density degree equal to a specific value of $d$. Our starting point is the following classification over number fields. For more on the structure of the potential density degree set of curves over number fields, see \cite{VV24}.

\begin{theorem}[Classification over number fields \cite{Fal83}, \cite{HS91}, \cite{AH91}, \cite{KV25}]\label{thm:geom_char_deg_d_nf}
    Let $K$ be a number field, and let $X/K$ be a smooth projective curve.
    \begin{enumerate}
        \item If $\min(\wp(X/K)) = d \le 3$, then $X_{\overline{K}}$ is a cover of degree $d$ of $\bbP^1$ or of an elliptic curve over $\overline{K}$.
        \item If $\min(\wp(X/K)) = d = 4,5$, then either $X_{\overline{K}}$ is a cover of degree $d$ of $\bbP^1$ or of an elliptic curve over $\overline{K}$, or $X_{\overline{K}}$ is (the normalization of) a DF curve of type $d$.
    \end{enumerate}
\end{theorem}

The theorem states that for $d\le 5$ the aforementioned possible reasons for having an infinite set of points of degree~$d$ are the \textit{only} way for a smooth projective curve over a number field to have $d=\min(\wp(X/K))$. It is a combination of results by Faltings ($d=1$), Harris and Silverman ($d=2$), Abramovich and Harris ($d=3$, and $d=4$ if the genus of the curve is not 7) and Kadets and Vogt ($d=4,5$).

In the present paper, we adapt this classification to smooth projective curves over function fields of characteristic 0. When $d=1$ and $K$ is a function field over an algebraically closed field $k$, Manin sketched a proof of such an adaption when $k=\mathbb{C}$ \cite{Man63} and Grauert later provided a complete proof \cite{Gra65}.

In the setting of function fields, we have to take into account \textit{constant} varieties.

\begin{definition}
    Let $K/k$ be a field extension and $X$ a variety over $K$. Then $X$ is called \textit{constant} or \textit{trivial} (over $k$) if it is defined over $k$, i.e., if there exists a variety $X_0$ over $k$ such that $X \cong (X_0)_{K}$.
\end{definition}

If a variety $X\cong(X_0)_K$ is constant, the $k$-points of $X_0$ induce $K$-points on $X$. In the following theorem, the purpose of the assumption in $(ii)$ that $X_{\overline{K}}$ is not dominated by a cover of degree at most $d$ of a constant curve is to exclude the situation where an infinitude of points of degree $d$ on $X$ come from the rational points of some constant variety, see also \Cref{sec:isotrivial_varieties}.

\begin{theoremintro}[\Cref{thm:main_result}]\label{thm:main_result_intro}
    Let $K$ be a function field over a field $k$ of characteristic 0, let $X/K$ be a smooth projective curve and set $d = \min(\wp(X/K))$. Assume that one of the following holds.
    \begin{enumerate}
        \item[(i)] $k = \mathbb{Q}$; or,
        \item[(ii)] $k$ is algebraically closed and $X_{\overline{K}}$ is not dominated by a cover of degree at most $d$ of a constant curve.
    \end{enumerate}
    Then the following statements hold.
    \begin{enumerate}
        \item If $d \le 3$, then $X_{\overline{K}}$ is a cover of degree $d$ of $\bbP^1$ or of an elliptic curve over $\overline{K}$.
        \item If $d = 4,5$, then either $X_{\overline{K}}$ is a cover of degree $d$ of $\bbP^1$ or of an elliptic curve over $\overline{K}$, or $X_{\overline{K}}$ is (the normalization of) a DF curve of type $d$.
    \end{enumerate}
\end{theoremintro}

We note that in \Cref{thm:main_result_intro} it is not possible to weaken the assumption \say{$X_{\overline{K}}$ is not dominated by a cover of degree at most $d$ of a constant curve} to \say{$X_{\overline{K}}$ is not a cover of degree at most $d$ of a constant curve}, as demonstrated by \Cref{ex:vb}.

The first step towards proving \Cref{thm:main_result_intro} is the following geometric characterization of smooth projective curves over function fields with minimal potential density degree equal to $d$. It is a function field analogue of \cite[Theorem~4.2]{BELOV19}, see also \Cref{rmk:abvar_over_nf}.

\begin{theoremintro}\label{thm:degree_d_points_from_abvar_intro}
    Let $K$ be a function field over a field $k$ of characteristic 0, let $X/K$ be a smooth projective curve, and let $d:=\min(\wp(X/K))$. Assume that one of the following holds:
    \begin{enumerate}
        \item[(i)] $k = \mathbb{Q}$ and $X_{\overline{K}}$ has gonality $> d$; or,
        \item[(ii)] $k$ is algebraically closed and $X_{\overline{K}}$ is not dominated by a cover of degree at most $d$ of a constant curve.
    \end{enumerate}
    Then there exists a diagram
    \begin{center}
        \begin{tikzcd}
	Z & {X_{\overline{K}}} \\
	{A,}
	\arrow[two heads, from=1-1, to=1-2]
	\arrow["{\textup{deg }d}"', from=1-1, to=2-1]
\end{tikzcd}

    \end{center}
    where $A$ is an abelian variety over $\overline{K}$ with $\dim(A) < \textup{max}(d,2)$, where $Z$ is a normal variety over $\overline{K}$, the morphism $Z\rightarrow A$ is finite of degree $d$ and $Z \rightarrow X_{\overline{K}}$ is surjective. Furthermore, there exists no such diagram where the degree of the morphism $Z\rightarrow A$ is smaller than $d$.
\end{theoremintro}

If $(i)$ is satisfied, \Cref{thm:degree_d_points_from_abvar_intro} is an application of the main result of \cite{Fal94}. In the second case, it is an application of the geometric Bombieri--Lang conjecture \cite[Conjecture~1.1]{XY23a} as proven for a large class of varieties (including closed subvarieties of abelian varieties) in \cite[Theorems~1.2~and~1.5]{XY23a}, \cite[Theorem~1.1]{XY23b} and \cite[Theorem~1.1]{Gao25}. We note that the formulation in \emph{loc.\ cit.\ }provides the precise statement we need, but that this statement can  presumably also be deduced from combining earlier results on the Bombieri--Lang conjecture \cite{Ray83}, \cite{Bui92}, \cite{Hru96}.

The proof of \Cref{thm:main_result_intro} is not obtained by an adaption of the proof of \Cref{thm:geom_char_deg_d_nf}, but rather, we reduce to the case over number fields. We achieve this as follows. The diagram we get from \Cref{thm:degree_d_points_from_abvar_intro} is defined over an affine $\mathbb{Q}$-scheme $S$ which is smooth over $\mathbb{Q}$. We consider specializations of this diagram over $S$ to closed points $s\in S$, which give diagrams over number fields. Moreover, the specialization of the model of $X$ over $S$ to $s$ has a potentially dense set of points of degree at most $d$ induced by the abelian variety in the diagram, hence we can apply \Cref{thm:geom_char_deg_d_nf}. The final step is to deduce the desired property of $X_{\overline{K}}$ from these specializations. For this, we use the following result.

\begin{theoremintro}[Lemmas \ref{lem:specialization_thm_genus_0}, \ref{lem:specialization_thm_genus_1} and \ref{lem:specialization_thm_DF}]\label{thm:specialization_result_intro}
    Let $S$ be a noetherian integral scheme over an algebraically closed field $k$, let $K$ be the function field of $S$, and let $X\rightarrow S$ be a smooth projective morphism such that the geometric fibers are curves. Fix a positive integer $d$, and consider the sets
    \begin{align*}
        \Sigma^d_0             & := \bigl\{ s\in S(k) \bigm| X_s\textup{ is a cover of degree }d\textup{ of }\mathbb{P}^1_{k} \bigr\},           \\
        \Sigma^d_1             & := \bigl\{s\in S(k) \bigm| X_s\textup{ is a cover of degree }d\textup{ of some elliptic curve over }{k}\bigr\}, \\
        \Sigma^d_{\textup{DF}} & :=        \Biggl\{s\in S(k) \Biggm|
        \vcenter{
            \setbox0 = \hbox{\textup{there exists an $m$ dividing $d$ such that $X_s$ is a}}
            \copy0
            \hbox to\wd0{\hfill \textup{cover of degree~$m$ of a DF curve of type $\frac{d}{m}$}\hfill}
        }
        \Biggr\}.
    \end{align*}
    Consider the geometric generic fiber $X_{\ol{K}}$. The following statements hold.
    \begin{enumerate}
        \item If $\Sigma^d_0$ is dense in $S$, then $X_{\ol{K}}$ is a cover of degree $d$ of $\mathbb{P}^1_{\overline{K}}$.
        \item If $\Sigma^d_1$ is dense in $S$, then $X_{\ol{K}}$ is a cover of degree $d$ of an elliptic curve over $\overline{K}$.
        \item If $\Sigma^d_{\textup{DF}}$ is dense in $S$, then there exists an $m>0$ dividing $d$ such that $X_{\overline{K}}$ is a cover of degree $m$ of a DF curve of type $\frac{d}{m}$.
    \end{enumerate}
\end{theoremintro}

The proofs of the three parts of \Cref{thm:specialization_result_intro} all follow a similar strategy. Namely, we first show that~$\Sigma^d_{\star}$ is contained in the image of a quasi-projective morphism $H\rightarrow S$ of schemes using the theory of Hom-schemes. The assumptions made in \Cref{thm:specialization_result_intro} imply that this morphism is dominant and thus that the generic point of $S$ is contained in the image of $H\rightarrow S$.

More precisely, the proof of part $(1)$ relies on the fact that the Hom-scheme $\homs^d_S(X,\mathbb{P}^1_S)$ that classifies morphisms $X\rightarrow \mathbb{P}^1_S$ of degree $d$, is quasi-projective over $S$ (see \cite[2.1]{Deb01} and \cite[Theorem~3.2]{Gro61}). Part $(2)$ admits a similar strategy, except that we have to replace $\mathbb{P}^1_S$ by a suitable universal family of elliptic curves.

For part $(3)$, we first define a functor that is represented by a quasi-projective scheme which plays a similar role as $\homs^d_S(X,\mathbb{P}^1_S)$ in the proof of part $(1)$, see \Cref{def:DF_functor}. We show the representability of this functor by a quasi-projective scheme in \Cref{lem:DF_functor_representable} using the decomposition of the Quot scheme.

\subsection{Conventions}
For $k$ a field, $\overline{k}$ denotes a fixed algebraic closure. A \textit{variety} $X$ over $k$ is a geometrically integral, separated scheme of finite type over $k$.

Let $X$ be a scheme. If $k$ is a field, and $x\in X(k)$, we sometimes write $x$ when we mean the image of the morphism $x$. Similarly, if $A\subseteq X(k)$, we sometimes interpret $A$ as a set of points in $X$, namely the images of the morphisms $\spec(k)\rightarrow X$ in $A$. In particular, when we say that $A$ is \textit{dense} in $X$, we mean that the union of the images of the morphisms in $A$ is a dense subset of $X$. The residue field of a point $x\in X$ is denoted by $\kappa(x)$.

Let $X$ be a scheme over a scheme $S$. If $T\rightarrow S$ is another scheme over $S$, we write $X_T = X \times_S T$ for the base change. Moreover, if $s\in S$, we write $X_s$ for the base change of $X$ along $\spec(\kappa(s)) \rightarrow S$. Finally,  if $Y$ is another scheme over $S$, and $f\colon X\rightarrow Y$ is a morphism over $S$, then for $s\in S$ we write $f_s$ for the induced morphism $X_s\rightarrow Y_s$ on fibers.

A \textit{cover of degree $d$} of an integral scheme $X$ is a finite morphism $Y\rightarrow X$, where $Y$ is a reduced scheme with irreducible components $Y_1, \ldots, Y_n$ such that each restriction $Y_i \rightarrow X$ is surjective and the sum of the degrees of the restrictions $Y_i \rightarrow X$ is equal to $d$.

If $S$ is a scheme, $\textup{Sch/S}$ denotes the category of schemes over $S$. The category of sets is denoted by $\textup{Sets}$.

Let $X$ be a curve over a field $k$. Then the \textit{gonality} of $X$ is given by \[\textup{gon}(X) = \min\left\{\deg(f)\bigm| f\colon X\rightarrow\bbP^1 \textup{ non-constant}\right\}.\]

\subsection{Acknowledgements}
I am grateful to Ariyan Javanpeykar for many insightful and encouraging discussions, for introducing me to this topic and for helpful feedback on an earlier version of this work. I thank Finn Bartsch for suggesting \Cref{ex:vb} and for helpful comments on an earlier draft. I also thank Ben Moonen for several useful remarks.

\section{Preliminaries}\label{sec:preliminaries}

\subsection{The symmetric product}\label{sec:symd}

Let $X/k$ be a smooth projective curve and $d$ a positive integer. Consider the $d$-th symmetric product $\sym^d(X)$ of $X$, which is the quotient of the $d$-fold product $X^d$ of $X$ over $k$ by the action of the symmetric group $S_d$ given by permuting the factors. As $\sym^d(X)$ parametrizes the effective divisors of degree $d$ on $X$ \cite[Theorem~3.13]{Mil86}, a rational point of $\sym^d(X)$ corresponds to a divisor $D = \sum_{i=0}^n m_iP_i$ on $X$ such that $\deg(D) = \sum_{i=0}^{n} m_i\cdot\deg(P_i) = d$. In particular, the set of points of degree at most $d$ on $X$ is infinite if and only if $\sym^d(X)(k)$ is infinite.

We note that there also exists a relative version of the symmetric product of a curve. Let $S$ be a scheme over a field $k$, and $X \rightarrow S$ a smooth projective morphism such that the geometric fibers are curves. Then $S_d$ acts on the $d$-fold product $X\times_S \dots \times_S X$ of $X$ over $S$ by permuting the factors. In this case, the geometric quotient $\sym^d_S(X)$ of $X\times_S \dots \times_S X$ by $S_d$ also exists \cite[Theorem~1, p.111]{Mum70}. Furthermore, $\sym^d_S(X)$ is naturally a scheme over $S$, and for all $s\in S$ we have a natural isomorphism $\sym^d_S(X)_s \cong \sym^d(X_s)$.

\subsection{Constant varieties} \label{sec:isotrivial_varieties}

Suppose $K$ is a function field over an algebraically closed field $k$. Let $X_0$ be a variety over $k$, and consider the constant variety $X := (X_0)_K$. Then $X$ has a dense set of rational points. Indeed, $X_0(k)$ is dense in $X_0$, and this set induces a dense set of $K$-points on $X$. Similarly, if $X_L$ is constant for some finite extension $L$ of $K$, then $X$ has a \emph{potentially} dense set of rational points.

More generally, any variety $X$ over $K$ such that $X_{\overline{K}}$ is dominated by a cover of degree at most $d$ of a constant variety has a potentially dense set of points of degree at most $d$. Indeed, then there exists a finite extension $L/K$ such that $X_L$ is dominated by a cover of degree at most $d$ of a constant variety, say~$Y$. Then the dense set of rational points of $Y$ can be pulled back to a dense set of points of degree at most $d$ on $X_L$.

In this case the origin of points of degree at most $d$ on $X$ is clear. Hence, in our results, we assume that $X_{\overline{K}}$ is not dominated by a degree at most $d$ cover of a constant variety. If $X$ is a curve, then by the following lemma, it suffices to assume that $X_{\overline{K}}$ is not dominated by a degree at most $d$ cover of a constant \textit{curve}.

\begin{lemma}\label{lem:dominated_deg_d_cover_variety_implies_curve}
    Let $K$ be a function field over a field $k$ of characteristic 0. Let $X$ be a curve over $K$ and suppose that $X_{\overline{K}}$ is dominated by a cover of degree $d$ of a constant variety over $\overline{K}$. Then $X_{\overline{K}}$ is dominated by a cover of degree at most $d$ of a constant curve over $\overline{K}$.
\end{lemma}

\begin{proof}
    Let $f\colon Z\rightarrow Y$ be a cover of degree $d$ of a constant variety $Y \cong (Y_0)_{\overline{K}}$ such that $Z$ dominates~$X_{\overline{K}}$. Since $K$ is of characteristic 0, the morphism $f$ is generically étale. Hence there exists a dense open $U\subseteq Y$ such that $f^{-1}U\rightarrow U$ is étale \cite[Theorem~3.2.1]{Poo17}. Let $P_0,\ldots,P_d$ be closed points in $f^{-1}U$ with pairwise distinct images in $X_{\overline{K}}$. Let $Q_0,\ldots,Q_d$ be the images of $P_0,\ldots,P_d$ in $Y_0$. Then there exists a curve $C_0\subset Y_0$ passing through the points $Q_0,\ldots,Q_d$. Write $C$ for $(C_0)_{\overline{K}}$ and $C'$ for the intersection $C\cap U$. Then $C'$ is a curve and $C$ is a constant curve.

    Consider the pullback $D$ of $Z$ along $C\subseteq Y$, and the pullback $D'$ of $D$ along $C'\subseteq C$. Then the resulting map $D' \rightarrow C'$ is finite étale, as $C'\subseteq U$. Furthermore, it is of degree $d$, because it is generically of degree $d$. This implies that $D'$ is the union of at most $d$ curves. Since we chose $d+1$ points, one of these curves, say $D''$, contains two of the points $P_i$. This implies that $D''\rightarrow X_{\overline{K}}$ is non-constant, and hence dominant. Let $\overline{D''}$ be the closure of $D''$ in $D$. Then the diagram
    \begin{center}
        \begin{tikzcd}
	{\overline{D''}} & {X_{\overline{K}}} \\
	C
	\arrow[from=1-1, to=1-2]
	\arrow["{\textup{deg}\le d}", from=1-1, to=2-1]
\end{tikzcd}
    \end{center}
    shows that $X_{\overline{K}}$ is dominated by a cover of degree at most $d$ of a constant curve.
\end{proof}

One can wonder, in the case that $X_{\overline{K}}$ is dominated by a cover of degree at most $d$ of a constant curve, whether $X_{\overline{K}}$ itself admits a morphism of degree at most $d$ to some constant curve (or variety). This is not true in general, as demonstrated by the following example.

\begin{example}\label{ex:vb}
    Let $B$ be a very general curve of genus $2$ over $\mathbb{C}$. Let $A$ be its Jacobian. Then $A$ is a simple principally polarized abelian variety over $\mathbb{C}$ and the Néron--Severi group of $A$ has rank 1. Let $\Theta$ be a divisor in the class that generates $\textup{NS}(A)$. Then $B\subseteq A$ is in the same class as $\Theta$.

    Let $d\ge 4$ and fix a function field $K/\mathbb{C}$. We now consider the base change $A_K$. By \cite[Proposition~5.15]{DF93}, a general curve $D\subseteq A_K$ in the linear system $|d\Theta_K|$ does not admit any morphisms of degree at most $2d$ to $\mathbb{P}^1$, and no dominant morphisms to any curve other than $\mathbb{P}^1$ and $D$ itself. In particular, it admits no morphism of degree at most $2d$ to a constant curve.

    There is a family of degree $2d$ divisors on $D$ parametrized by $A_K$. Indeed, let $a\in A_K$ be a point and consider the divisor $a + \Theta_K$ on $A$. Then $a + \Theta$ is algebraically equivalent to $\Theta_K$, and hence $(a + \Theta_K)\cdot D = 2d$. This means that the restriction of $a + \Theta_K$ to $D$ is a divisor of degree $2d$ on $D$. This family gives us an injective map $A_K \rightarrow \sym^{2d}(D)$ (this is the map $\psi$ in \cite[4.3]{DF93}).

    Consider the composition $B_K \rightarrow A_K \rightarrow \sym^{2d}(D)$, and the following diagram, where the square is cartesian.
    \begin{center}
        \begin{tikzcd}
	& X & {D\times\textup{Sym}^{2d-1}(D)} & D \\
	{} & {B_K} & {\textup{Sym}^{2d}(D)}
	\arrow[from=1-2, to=1-3]
	\arrow[from=1-2, to=2-2]
	\arrow[from=1-3, to=1-4]
	\arrow["{\textup{deg } 2d}"', from=1-3, to=2-3]
	\arrow[from=2-2, to=2-3]
\end{tikzcd}
    \end{center}
    Then $X\rightarrow B_K$ is of degree $2d$. Furthermore, $X\rightarrow D$ is dominant, because $B_K\rightarrow \sym^{2d}(D)$ has an image of positive dimension. We find that $D$ is dominated by a cover of degree $2d$ of the constant curve~$B_K$, while there exist no maps of degree at most $2d$ from $D$ to a constant curve.
\end{example}

\subsection{Debarre--Fahlaoui curves} \label{sec:DF-curves}

We give the definition of Debarre--Fahlaoui curves, which were first constructed by Debarre and Fahlaoui \cite[Section~4]{DF93} to provide counterexamples to the conjecture of Abramovich and Harris \cite{AH91} stating that all potentially dense sets of points of degree $d$ on smooth projective curves over number fields can by geometrically explained by a morphism of degree $d$ to a curve of genus 0 or 1. An introduction to Debarre--Fahlaoui curves is also given in Section~5 of \cite{KV25}.

Let $E$ be an elliptic curve over a field $k$ with neutral element $0\in E(k)$. Then the addition map $\pi\colon\sym^2(E) \rightarrow E$ gives $\sym^2(E)$ the structure of a ruled surface. Let $H$ be the divisor in $\sym^2(E)$ consisting of all effective divisors of degree $2$ on $E$ containing $0$, and let $F$ be the divisor corresponding to the fiber of $\pi$ over $0$. Then $H$ and $F$ span all numerical equivalence classes of divisors, and we have $H^2 = 1$, $H\cdot F = 1$ and $F^2 = 0$.

\begin{definition}\label{def:DF_curve}
    Let $X$ be a projective curve over a field $k$, and let $d$ be a positive integer. Then $X$ is a \textit{Debarre--Fahlaoui curve of type $d$} (or \textit{DF curve of type $d$} for short), if there exists an elliptic curve $E$ over~$k$ and a closed immersion $X \hookrightarrow \sym^2(E)$ such that the numerical equivalence class of $X$ considered as a curve in $\sym^2(E)$ is equal to $(d+m)H- mF$ for some $1\le m\le d$.
\end{definition}

\begin{remark}
    Debarre and Fahlaoui first considered DF curves \cite{DF93}, but only for the case $m=1$. In \cite{KV25}, Kadets and Vogt use the more general definition that we have given here.
\end{remark}

Suppose $X \subset \sym^2 E$ is a DF curve of type $d$. Then there is a family of effective divisors of degree $d$ on~$X$ parametrized by $E$, see the discussion below Definition 5.1 in \cite{KV25}. Thus, if $E$ has a potentially dense set of rational points, which is true if $k$ is a number field or function field, then $X$ has a potentially dense set of points of degree at most $d$.

Moreover, Debarre and Fahlaoui show that certain DF curves of type $d$ do not admit morphisms of degree $d$ to curves of genus 0 or 1 \cite[Propositions~5.7~and~5.14]{DF93}, disproving the conjecture of Abramovich and Harris. Kadets and Vogt extend this \cite[Proposition~5.2]{KV25}.

\vspace{6pt}

We now consider a relative setting. Let $k$ be a field, $U$ a scheme over $k$, and $\mathcal{E}$ an elliptic curve over $U$, i.e., $\mathcal{E}/U$ is a smooth proper group scheme with connected one-dimensional geometric fibers. Then there exists a section $\sigma_0\colon U \rightarrow \mathcal{E}$ such that for all $\lambda \in U$, the fiber $\mathcal{E}_{\lambda}$ is an elliptic curve with neutral element $\sigma_0(\lambda)$.

For each $\lambda \in U$, we have the divisors $H$ and $F$ on $\sym^2(\mathcal{E}_{\lambda})$ as defined above \Cref{def:DF_curve}. To make a distinction between these divisors for varying $\lambda$, we write $H_{\lambda}$ for $H$ and $F_{\lambda}$ for $F$ in this setting.

We are going to construct line bundles $\mathcal{H}$ and $\mathcal{F}$ on $\sym^2_U(\mathcal{E})$ such that for all $\lambda \in U$, divisors corresponding to the pulled back line bundles $\mathcal{H}_{\lambda}$ and $\mathcal{F}_{\lambda}$ are of rational equivalence class $H_{\lambda}$ and~$F_{\lambda}$ respectively. We achieve this as follows. Let $\Sigma_0 \subseteq \mathcal{E}$ be the divisor associated to the ideal sheaf of $\sigma_0$, and let $p_i\colon \mathcal{E}\times_U\mathcal{E} \rightarrow \mathcal{E}$ denote the projection onto the $i$-th coordinate. Furthermore, let $\varphi\colon\mathcal{E}\times_U\mathcal{E} \rightarrow \sym^2_U(\mathcal{E})$ be the geometric quotient map, and finally let $\pi\colon\mathcal{E}\times_U\mathcal{E} \rightarrow \mathcal{E}$ be given by the group law of $\mathcal{E}/U$. We define the line bundles $\mathcal{H}$ and $\mathcal{F}$ on $\sym^2_U(\mathcal{E})$ to be those associated to the divisors $\varphi_{*}p_1^*(\Sigma_0)$ and $\pi^*(\Sigma_0)$ respectively. Then $\mathcal{H}$ and $\mathcal{F}$ satisfy the above property.

\section{\texorpdfstring{Points of degree $d$ induced by abelian varieties}{Points of degree d points induced by abelian varieties}}\label{sec:abvar}

The goal of this section is to give a proof of \Cref{thm:degree_d_points_from_abvar_intro}. Let $K$ be a function field of characteristic~0 and $X/K$ a smooth projective curve. Set $d:=\min(\wp(X/K))$. We assume that one of the following holds.
\begin{enumerate}
    \item[$(i)$] $k = \mathbb{Q}$ and $\gon(X_{\overline{K}}) > d$; or,
    \item[$(ii)$] $k$ is algebraically closed and $X_{\overline{K}}$ is not dominated by a cover of degree at most $d$ of a constant curve.
\end{enumerate}
We first discuss case $(i)$. We then focus on the proof of \Cref{thm:degree_d_points_from_abvar_intro} in the second case. However, the proof for the latter case can be adapted to also work for case $(i)$, as explained at the start of the proof of case $(ii)$.

\begin{proof}[Proof of \Cref{thm:degree_d_points_from_abvar_intro} in case $(i)$]\label{rmk:abvar_over_nf}
    Consider the first case of \Cref{thm:degree_d_points_from_abvar_intro}, i.e., $K$ is a function field over $\mathbb{Q}$ and $\gon(X_{\overline{K}})>d$. Observe that \cite[Theorem~4.2]{BELOV19} easily generalizes from number fields to all fields which are finitely generated over $\mathbb{Q}$, because Faltings' theorem \cite{Fal94} generalizes to this setting as well. It now follows directly from this generalization of \cite[Theorem~4.2(1)]{BELOV19} that there exists a point $x\in X$ of degree $d$ that is not AV-isolated (see \cite[Definition~4.1]{BELOV19}), which is equivalent to the existence of a diagram as in \Cref{thm:degree_d_points_from_abvar_intro}.
\end{proof}

Assume that $X$ has an effective divisor $D$ of degree $d$, and identify the points of $\sym^d(X)$ with effective divisors of degree $d$ on $X$. Then the Abel-Jacobi map is given by
\[
    \varphi_{D}\colon \sym^d(X)  \rightarrow \textup{Jac}(X), \quad
    D'                                 \mapsto [D' - D],
\]
where $\textup{Jac}(X)$ denotes the Jacobian of $X$, and the brackets denote the linear equivalence class of a divisor. Note that the fibers of $\varphi_D$ are precisely complete linear systems of degree $d$. We find that $\varphi_D$ is injective if $\textup{gon}(X_{\overline{K}}) > d$ (which is implied by the assumption in $(ii)$), and thus $\varphi_D$ is a closed immersion in this case by \sp{04DG}, \sp{0B2G} and \cite[Theorem~5.1(b)]{Mil86}.

We then get a useful description of the rational points of $\sym^d(X)$ from the geometric Bombieri--Lang conjecture \cite[Conjecture~1.1]{XY23a}, which is proved in the case of subvarieties of abelian varieties by \cite{Ray83}, \cite{Bui92} and \cite{Hru96}. For a more explicit proof of this fact, as well as an extension of this result, see \cite[Theorem~1.1]{Gao25}. For convenience, we refer to \cite[Theorem~1.1]{Gao25} in the proof of \Cref{thm:degree_d_points_from_abvar_intro}.

\begin{proof}[Proof of \Cref{thm:degree_d_points_from_abvar_intro} in case $(ii)$]
    We assume that $(ii)$ holds. Alternatively, if $(i)$ holds, the statement can be proven analogously to the proof that follows, except that it requires an application of Faltings' theorems \cite{Fal94} and \cite{Fal83}, instead of \cite[Theorem~1.1]{Gao25} and the main result of \cite{Gra65}, respectively.

    If $d=1$, the main result of \cite{Gra65} implies that the genus of $X$ is 0 or 1. However, $\textup{gon}(X_{\overline{K}})>1$, so~$X$ must be of genus 1. This means that $X_{\overline{K}}$ can be given the structure of an elliptic curve. Hence, we can take $A = Z = X_{\overline{K}}$.

    Now assume $d>1$. By replacing $K$ with a finite extension, we may assume that $X(K)\neq\emptyset$, and that~$X$ has infinitely many points of degree $d$. Write $\textup{Sp} := \textup{Sp}_{\textup{alg}}(\sym^d(X))$ for the \emph{algebraic special locus} of $\sym^d(X)$, i.e., the Zariski closure in $\sym^d(X)$ of the union of all images of non-constant rational maps $A\dashrightarrow \sym^d(X)_{\overline{K}}$ with $A$ an abelian variety over $\overline{K}$. We are going to prove that $\textup{Sp}$ is non-empty.

    It follows from \cite[Theorem~1.1]{Gao25} that there are integers $r,s\geq 0$ such that
    \[
        \overline{\left(\sym^{d}(X)\setminus\textup{Sp}\right)(K)} = Z_1 \cup \ldots \cup Z_r \cup \{p_1,\ldots,p_s\}, \tag{$\ast$}
    \]
    where each $Z_i$ is a positive dimensional subvariety of $\sym^{d}(X)$ with constant normalization and each $p_i$ a closed point. We claim that $r=0$.

    We argue by contradiction. Suppose that $r\ge 1$. Consider the diagram
    \begin{center}
        \begin{tikzcd}
	W & {X_{\overline{K}}\times \textup{Sym}^{d-1}(X_{\overline{K}})} & {X_{\overline{K}}} \\
	{Z_{1,\overline{K}}} & {\textup{Sym}^d(X_{\overline{K}}),} 
	\arrow[from=1-1, to=1-2]
	\arrow["{\textup{deg }d}"', from=1-1, to=2-1]
	\arrow[from=1-2, to=1-3]
	\arrow["{\textup{deg }d}", from=1-2, to=2-2]
	\arrow[from=2-1, to=2-2]
\end{tikzcd}
    \end{center}
    where the square is cartesian. Then $W\rightarrow X_{\overline{K}}$ is surjective. Indeed, if it were not, the image would consist of finitely many points, which would imply that the image of the closed immersion $Z_{1,\overline{K}} \rightarrow \sym^d(X_{\overline{K}})$ consists of finitely many points as well, contradicting the positive dimensionality of $Z_{1,\overline{K}}$.

    Let $Z'_{1,\overline{K}}$ be the normalization of $Z_{1,\overline{K}}$. Define $W'$ to be the pullback of $W$ along the normalization map $Z'_{1,\overline{K}} \rightarrow Z_{1,\overline{K}}$. As $Z'_{1,\overline{K}}$ is a constant variety, $W'$ is a cover of degree $d$ of a constant variety dominating~$X_{\overline{K}}$. By \Cref{lem:dominated_deg_d_cover_variety_implies_curve}, this means that there also exists a cover of degree at most $d$ of a constant curve dominating $X_{\overline{K}}$, which directly contradicts our assumption on $X_{\overline{K}}$. We conclude that $r=0$.

    This shows that ($\ast$) is finite. However, $\sym^{d}(X)$ has infinitely many $K$-points, because the set of points of degree $d$ on $X$ is infinite (see \Cref{sec:symd}). Thus, $\textup{Sp}$ is indeed non-empty. This means that we can find a non-constant rational map $A \dashrightarrow \sym^d(X)_{\overline{K}}$ with $A$ an abelian variety over $\overline{K}$.

    As $\sym^d(X)$ is isomorphic to a subvariety of $\textup{Jac}(X)$, $A$ is smooth over $\overline{K}$ and $\textup{Jac}(X)$ has no rational curves \cite[Proposition~3.9]{Mil08}, Lemmas~3.2~and~3.5 of \cite{JK20} imply that the rational map $A \dashrightarrow \sym^d(X)_{\overline{K}}$ extends to a morphism $A \rightarrow \sym^d(X)_{\overline{K}}$. The image of $A \rightarrow \sym^d(X)_{\overline{K}}$ is isomorphic to an abelian subvariety $A'$ of $\textup{Jac}(X)$. Hence, by replacing $A$ with $A'$, we may assume that $A \rightarrow \sym^d(X)_{\overline{K}}$ is a closed immersion.\newpage

    Consider the diagram
    \begin{center}
        \begin{tikzcd}
	Z & {X_{\overline{K}}\times \textup{Sym}^{d-1}(X_{\overline{K}})} & {X_{\overline{K}}} \\
	A & {\textup{Sym}^d(X_{\overline{K}}),}
	\arrow[from=1-1, to=1-2]
	\arrow["{\textup{deg }d}"', from=1-1, to=2-1]
	\arrow[from=1-2, to=1-3]
	\arrow["{\textup{deg }d}", from=1-2, to=2-2]
	\arrow[from=2-1, to=2-2]
\end{tikzcd}
    \end{center}
    where the square is cartesian. Then $Z\rightarrow X_{\overline{K}}$ must be surjective, as otherwise the image would consist of only finitely many points, contradicting that the image of $A \rightarrow \sym^d(X)_{\overline{K}}$ is positive dimensional. Furthermore, $\dim(A) \le d$ is immediate. Since $d < \textup{gon}(X_{\overline{K}}) \le \dim(\textup{Jac}(X))$ (as $g(X) > 1$), we see that $\sym^d(X_{\overline{K}}) \neq A$, and thus $\dim(A) < d$.

    We may replace $Z$ by a normalization in the above diagram to get that $Z$ is normal. To see that $Z$ is irreducible, note that if $Z$ were not irreducible, we could replace $Z$ by an irreducible component to get a diagram as in the theorem where the degree of the morphism $Z\rightarrow A$ smaller than $d$. Therefore it remains to show that there exists no such diagram.

    Suppose by contradiction such a diagram does exist, where the morphism $Z\rightarrow A$ is of degree $d'<d$. The diagram is defined over some finite extension $K'$ of $K$. Let $A'$ be the abelian variety over $K'$ in this diagram. It follows from the main result of \cite{FJ74} that $A'$ has a potentially dense set of rational points. These rational points can be pulled back to a potentially dense set of points of degree at most $d'$ on $X$, which contradicts $d = \min(\wp(X/K))$.
\end{proof}

\section{Specialization lemmas}\label{sec:specialization_lemmas}

In this section, we prove \Cref{thm:specialization_result_intro}. Throughout, all fiber products $\--\times\--$ without subscript are over the base field.

\begin{lemma}\label{lem:specialization_thm_genus_0}
    Let $S$ be a noetherian integral scheme over an algebraically closed field $k$, let $K$ be the function field of $S$, and let $X\rightarrow S$ be a smooth projective morphism such that the geometric fibers are curves. Fix a positive integer $d$. Suppose that
    \[
        \Sigma^d_0 := \bigl\{s\in S(k) \bigm| X_s \textup{ is a cover of degree }d\textup{ of }\bbP^1_k \bigr\}
    \]
    is dense in $S$. Then the geometric generic fiber $X_{\overline{K}}$ is a cover of degree $d$ of $\bbP^1_{\overline{K}}$.
\end{lemma}

\begin{proof}
    By the theory of Hilbert schemes \cite{Gro61}, the functor $(\textup{Sch}/S)^{\textup{op}} \rightarrow \textup{Sets}$ given by
    \begin{align*}
        T \mapsto \bigl\{ f\colon X_T \rightarrow \mathbb{P}^1_T \bigm| \textup{for all } t\in T, \; f_t \textup{ is finite of degree }d \bigr\}
    \end{align*}
    is representable by a quasi projective scheme $H:= \homs^d_S(X,\mathbb{P}_S^1)$ over $S$, see also \cite[Chapter~2]{Deb01}.

    We claim that $H\rightarrow S$ is dominant. Suppose $s\in \Sigma^d_0$, i.e., there exists a finite morphism $X_s \rightarrow \bbP^1_k$ of degree $d$. Such a morphism corresponds precisely to a $k$-point of the fiber $H_s$. In particular, $H_s$ is non-empty, i.e., $s$ lies in the image of $H\rightarrow S$. We find that $\Sigma^d_0$ is contained in the image of $H\rightarrow S$. The claim now follows from the assumption that $\Sigma^d_0$ is dense in $S$.

    Since $H\to S$ is dominant and quasi-compact (even quasi-projective), the generic point $\eta$ of $S$ is in the image of $H\to S$ \sp{01RL}. In particular, $H_{\eta}$ is non-empty. Then, since $H_{\eta}$ is a non-empty quasi-projective scheme over $K$, it must have a $\overline{K}$-point. Such a point corresponds precisely to a morphism $X_{\overline{K}}\rightarrow \bbP^1_{\overline{K}}$ of degree $d$.
\end{proof}

\begin{lemma}\label{lem:specialization_thm_genus_1}
    Let $S$ be a noetherian integral scheme over an algebraically closed field $k$ with function field $K$, and let $X\rightarrow S$ be a smooth projective morphism such that the geometric fibers are curves. Fix a positive integer $d$. Suppose that
    \[
        \Sigma^d_1 := \bigl\{s\in S(k) \bigm| X_s \textup{ is a cover of degree }d\textup{ of some elliptic curve over }k \bigr\}
    \]
    is dense in $S$. Then the geometric generic fiber $X_{\overline{K}}$ is a cover of degree $d$ of some elliptic curve over $\overline{K}$.
\end{lemma}

We stress that in \Cref{lem:specialization_thm_genus_1}, for $s\neq s' \in \Sigma^d_1$, the elliptic curves of which $X_s$ and $X_{s'}$ are covers of degree $d$, are allowed to be distinct.

\begin{proof}
    Let $U/k$ be a smooth curve and let $\mathcal{E}\rightarrow U$ be an elliptic curve such that for every elliptic curve~$E$ over $k$, there exists a point $\lambda\in U(k)$ such that $\mathcal{E}_{\lambda} \cong E$. (For example, if $\textup{char}(k)\neq 2$, we can take $U = \mathbb{A}^1_k\setminus\{0,1\}$ and $\mathcal{E}$ the \textit{Legendre elliptic curve} \cite[Proposition~III.1.7(a)]{Sil09}. If $\textup{char}(k)= 2$, we can take $U$ to be any modular curve $Y(p)$ for $p$ any odd prime number.)

    There exists a quasi-projective scheme $H:= \homs^d_{U\times S}(U\times X, \mathcal{E}\times S)$ over $U\times S$ representing the functor $(\textup{Sch}/(U\times S))^{\textup{op}} \rightarrow \textup{Sets}$ given by
    \begin{align*}
        T \mapsto \bigl\{ f\colon (U\times X)_T \rightarrow (\mathcal{E}\times S)_T \bigm| \textup{for all } t\in T, \; f_t \textup{ is finite of degree }d  \bigr\}.
    \end{align*}

    Suppose $s\in \Sigma^d_1$, so that there exists an elliptic curve $E$ over $k$ and a finite morphism $X_s\rightarrow E$ of degree~$d$. Then $E\cong \mathcal{E}_{\lambda}$ for some $\lambda\in U(k)$. Consider the fiber $H_{\lambda,s}$ of $H$ over the point $(\lambda,s)\in (U\times S)(k)$. Then the morphism $X_s\rightarrow E$ corresponds to a $k$-point of $H_{\lambda,s}$. Since $H_{\lambda,s}$ maps to the fiber $H_s$, we find that $H_s$ is non-empty. As this is true for all $s\in\Sigma^d_1$, we conclude that $\Sigma^d_1$ must be contained in the image of $H\rightarrow S$, and hence that this morphism is dominant by our assumption that $\Sigma_1^d$ is dense.

    Since $H\to S$ is dominant and quasi-compact, the generic point $\eta$ of $S$ is in its image \sp{01RL} and thus $H_{\eta}$ is non-empty. Since $H_{\eta}$ is a  non-empty quasi-projective scheme over $U_{K}$, there is a $\overline{K}$-point~$\mu$ of $U_{K}$ which is in the image of the morphism $H_{\eta}\rightarrow U_{K}$, i.e., the fiber $H_{\eta,\mu}$ is non-empty. Then, since $H_{\eta,\mu}$ is a non-empty quasi-projective scheme over $\overline{K}$, it must have a $\overline{K}$-point itself. Such a point corresponds precisely to a morphism $X_{\overline{K}}\rightarrow \mathcal{E}_{\overline{K},\mu}$ of degree $d$.
\end{proof}

For the final specialization lemma, concerning DF curves, we rely on the representability of the following functor by a quasi-projective scheme (see \Cref{lem:DF_functor_representable} below).

\begin{definition}\label{def:DF_functor}
    Let $S$ be a noetherian integral scheme over an algebraically closed field $k$, and $X\rightarrow S$ a smooth projective morphism such that the geometric fibers are curves. Let $U$ be a scheme over $k$, and $\mathcal{E}$ an elliptic curve over $U$. For a positive integer $d$, we define
    \[
        \textup{DF}^d_{S/k,\mathcal{E}/U}(X)\colon (\textup{Sch}/(U\times S))^{\textup{op}} \rightarrow \textup{Sets}
    \]
    to be the functor given by
    \begin{align*}
        T \mapsto
        \Vastl\{  f\colon & X_T\rightarrow (\sym^2_U\mathcal{E})_T \Vastm|
        \vcenter{
            \setbox0 = \hbox{for all $t \in T$ there exists an $m$ dividing $d$ such that $f_t\colon X_t\rightarrow \textup{Im}(f_t)$}
            \copy0
            \hbox to\wd0{\hfill is a finite morphism of degree $m$ and the closed immersion \hfill}
            \hbox to\wd0{\hfill $\im(f_t)\hookrightarrow \sym^2 \mathcal{E}_t$ gives $\im(f_t)$ the structure of a DF curve of type $\frac{d}{m}$\hfill}
        }
        \Vastr\}.
    \end{align*}
\end{definition}

If $k$ and $U$ are clear from the context, we write $\textup{DF}^d_{S,\mathcal{E}}(X)$ instead of $\textup{DF}^d_{S/k,\mathcal{E}/U}(X)$. Moreover, if $S = \spec(R)$ is affine, we write $R$ instead of $\spec(R)$ in the subscript of $\textup{DF}^d_{S/k,\mathcal{E}/U}(X)$.

\begin{lemma} \label{lem:DF_functor_representable}
    Let $S$ be a noetherian integral scheme over an algebraically closed field $k$, and $X\rightarrow S$ a smooth projective morphism such that the geometric fibers are curves. Let $U$ be a scheme over $k$, and $\mathcal{E}$ an elliptic curve over $U$. Then the functor $\textup{DF}^d_{S,\mathcal{E}}(X)$ is representable by a quasi-projective scheme over $U\times S$, which we denote by $\textup{\underline{DF}}^d_{S,\mathcal{E}}(X)$.
\end{lemma}

\begin{proof}
    Consider the scheme
    \[
        \homs_{U\times S}(U\times X, \sym^2_U\mathcal{E}\times S)
    \]
    over $U\times S$ representing the Hom-functor $T \mapsto \hom_T((U\times X)_T, (\sym^2_U\mathcal{E}\times S)_T)$ \cite[Theorem~I.1.10]{Kol96}. Then the functor
    \[
        T \mapsto \{ f\in \hom_T((U\times X)_T, (\sym^2_U\mathcal{E}\times S)_T) \;|\; f\textup{ is a finite morphism} \}
    \]
    is an open subfunctor of the Hom-functor by \cite[Proposition~12.93]{GW10}, and hence representable by an open subscheme of $\homs_{U\times S}(U\times X, \sym^2_U\mathcal{E}\times S)$. Let us write $\underline{\textup{Fin}}$ for this scheme (we suppress $X$, $U$, $\mathcal{E}$ and $S$ from the notation for simplicity).

    For $L$ a relatively very ample line bundle on $(U\times X)\times_{U\times S} (\sym^2_U\mathcal{E}\times S) = X \times \sym^2_U\mathcal{E}$, we get a decomposition
    \[
        \underline{\textup{Fin}} = \coprod_{\Phi\in\bbQ[Y]} \underline{\textup{Fin}}^{\Phi,L},
    \]
    where each $\underline{\textup{Fin}}^{\Phi,L}$ is a quasi-projective scheme over $U\times S$ representing the functor
    \begin{align*}
        T \mapsto
        \Bigg\{ f\in \underline{\textup{Fin}} \Biggm|
        \vcenter{
        \setbox0 = \hbox{for all $t\in T$ the Hilbertpolynomial $\Phi_t$ of $\Gamma_{f_t}\subseteq X_t \times \sym^2 \mathcal{E}_t$}
        \copy0
        \hbox to\wd0{\hfill with respect to $L|_{\Gamma_{f_t}}$ is equal to $\Phi$ \hfill}
        }
        \Bigg\},
    \end{align*}
    see \cite[Theorem~3.2]{Gro61}, and \cite[Theorem~6.3,~p.4]{Nit05}.

    Let $\mathcal{M}$ be a line bundle on $X$ which is ample relative to $X\rightarrow S$. Then for each $t\in S$, we get an ample line bundle $\mathcal{M}_t := \mathcal{M}|_{X_t}$. Moreover, $t \mapsto \deg(\mathcal{M}_t)$ is constant. We write $n$ for this degree. Let $p_1\colon X\times \sym^2_U(\mathcal{E}) \rightarrow X$ and $p_2\colon X\times \sym^2_U(\mathcal{E}) \rightarrow \sym^2_U(\mathcal{E})$ denote the projections, and define the following line bundles on $X\times \sym^2_U(\mathcal{E})$:
    \begin{align*}
        \mathcal{L}_{\mathcal{H}}                    & := p_1^* \mathcal{M} \otimes p_2^*\mathcal{H};                                   \\
        \mathcal{L}_{\mathcal{H}\otimes \mathcal{F}} & := p_1^* \mathcal{M} \otimes p_2^* \left(\mathcal{H} \otimes \mathcal{F}\right),
    \end{align*}
    where $\mathcal{H}$ and $\mathcal{F}$ are the line bundles on $ \sym^2_U(\mathcal{E})$ defined in \Cref{sec:DF-curves}.

    For all $\lambda \in U$, the restrictions $\mathcal{H}_{\lambda}$ and $(\mathcal{H\otimes F})_{\lambda}$ are ample by \cite[Proposition~V.2.21(b)]{Har77}. It then follows from the Segre embedding and \cite[Theorem~1.7.8]{Laz04I} that $\mathcal{L}_{\mathcal{H}}$ and $\mathcal{L}_{\mathcal{H}\otimes \mathcal{F}}$ are relatively ample line bundles on $X\times \sym^2_U(\mathcal{E})$ with respect to the map $X\times \sym^2_U(\mathcal{E}) \rightarrow U\times S$.

    Fix $j \gg 0$ such that $\mathcal{L}_{\mathcal{H}}^{\otimes j}$ and $\mathcal{L}_{\mathcal{H}\otimes\mathcal{F}}^{\otimes j}$ are relatively very ample. Let $g$ denote the genus of the fibers $X_s$ of $X \rightarrow S$. Let $T$ be a scheme over $U\times S$ and $f\in \underline{\textup{Fin}}(T)$, i.e., $f$ corresponds to a finite morphism $X_T \rightarrow (\sym^2_U\mathcal{E})_T$. For $i=0,1$, and $t\in T$, the Hilbert polynomial $\Phi_i(Y) \in \mathbb{Q}[Y]$ of the graph $\Gamma_{f_t} \subseteq X_t \times \sym^2\mathcal{E}_t$ with respect to the line bundles $\mathcal{L}_0 := \mathcal{L}_{\mathcal{H}}^{\otimes j}$ and $\mathcal{L}_1 := \mathcal{L}_{\mathcal{H}\otimes\mathcal{F}}^{\otimes j}$ is given by
    \[
        \Phi_i(Y) = \deg(\mathcal{L}_{i,t}|_{\Gamma_{f_t}}) \cdot Y +
        (1 - g).
    \]
    This follows from the Riemann--Roch theorem, see \cite[Remark~1.3.1]{Har77}. Let $aH_t + bF_t$ be the numerical equivalence class of $\im(f_t)$ inside $\sym^2\mathcal{E}_t$ and $\deg(f_t)$ the degree of the finite morphism $f_t\colon X_t\rightarrow \im(f_t)$. Using the projection formula \cite[Proposition~1.10]{Deb01}, we find
    \begin{align*}
        \deg\left((\mathcal{L}_{i,t})|_{\Gamma_{f_t}}\right) & = \Gamma_{f_t} \cdot \mathcal{L}_{i,t}                                                                                                                                \\
                                                             & = \Gamma_{f_t} \cdot p^*_{X_t}(\mathcal{M}_t)^{\otimes j} + \Gamma_{f_t} \cdot p^*_{\sym^2\mathcal{E}_t}(\mathcal{H}_t \otimes \mathcal{F}_t^{\otimes i})^{\otimes j} \\
                                                             & = j\cdot n + j \cdot \deg(f_t) \cdot ((aH_t + bF_t) \cdot (H_t+iF_t))                                                                                                 \\
                                                             & = j\cdot n + j \cdot \deg(f_t) \cdot (a(i+1)+b).
    \end{align*}
    We conclude that there exists an $m>0$ dividing $d$ such that $f_t$ is a finite morphism of degree $m$ and the closed immersion $\im(f_t)\hookrightarrow\sym^2\mathcal{E}_t$ gives $\im(f_t)$ the structure of a DF curve of type $\frac{d}{m}$ if and only if there exists $0\le l \le d$ such that
    \begin{align*}
        \Phi_0(Y) & = j(n+d)Y + (1-g), \textup{ and} \\
        \Phi_1(Y) & = j(n+2d+l)Y + (1-g).
    \end{align*}
    Write $P(Y) =  j(n+d)Y + (1-g)$ and $Q_l(Y) = j(n+2d+l)Y + (1-g)$. It follows that
    \[
        \textup{DF}^d_{S,\mathcal{E}}(X)(T) = \left(\underline{\textup{Fin}}^{P(Y),\mathcal{L}_0} \cap \coprod_{l=0}^d  \underline{\textup{Fin}}^{Q_l(Y),\mathcal{L}_1}\right)(T).
    \]
    Hence, the functor $\textup{DF}^d_{S,\mathcal{E}}(X)$ is represented by a quasi-projective scheme over $U\times S$.
\end{proof}

\begin{lemma}\label{lem:specialization_thm_DF}
    Let $S$ be a noetherian integral scheme over an algebraically closed field $k$, and let $X\rightarrow S$ be a smooth projective morphism such that the geometric fibers are curves. Let $K$ be the function field of $S$. Fix a positive integer $d$. Suppose that
    \begin{align*}
        \Sigma^d_{\textup{DF}} := \Biggl\{ s\in S(k) \Biggm|
        \vcenter{
            \setbox0=\hbox{\textup{There exists an $m$ dividing $d$ such that $X_s$ is}}
            \copy0
            \hbox to\wd0{\hfill \textup{a cover of degree $m$ of a DF curve of type $\frac{d}{m}$}\hfill}
        }
        \Biggr\}
    \end{align*}
    is dense in $S$. Then there is an $m>0$ dividing $d$, such that the geometric generic fiber $X_{\overline{K}}$ admits a morphism of degree $m$ to a DF curve of type $\frac{d}{m}$.
\end{lemma}

\begin{proof}
    As in the proof of \Cref{lem:specialization_thm_genus_1}, let $U$ be a smooth curve over $k$ and let $\mathcal{E}\rightarrow U$ be an elliptic curve such that for every elliptic curve $E$ over $k$, there exists a point $\lambda\in U(k)$ such that $\mathcal{E}_{\lambda} \cong E$. By \Cref{lem:DF_functor_representable}, there exists a quasi-projective scheme $\underline{\textup{DF}}^d_{S,\mathcal{E}}(X)$ over $U\times S$ representing the functor $\textup{DF}^d_{S,\mathcal{E}}(X)$. We are going to show that $\Sigma^d_{\textup{DF}}$ is contained in the image of the composition $\varphi\colon\underline{\textup{DF}}^d_{S,\mathcal{E}}(X) \rightarrow U\times S \rightarrow S$.

    Suppose $s\in\Sigma^d_{\textup{DF}}$, i.e., there exists an $m>0$ dividing $d$, and a morphism $f\colon X_s\rightarrow C$ of degree~$m$ with $C$ a DF curve of type $\frac{d}{m}$. Then $C \subseteq \sym^2\mathcal{E}_{\lambda}$ for some $\lambda\in U(k)$ such that the composition $X_s\rightarrow C \hookrightarrow \sym^2\mathcal{E}_{\lambda}$ corresponds to a $k$-point of the scheme $\underline{\textup{DF}}^d_{k,\mathcal{E}_{\lambda}}(X_s)$. In particular, $\underline{\textup{DF}}^d_{k,\mathcal{E}_{\lambda}}(X_s)$ is non-empty.

    Note that $\underline{\textup{DF}}^d_{k,\mathcal{E}_{\lambda}}(X_s)$ is the pullback of $\underline{\textup{DF}}^d_{S,\mathcal{E}}(X)$ along $(\lambda, s) \in (U\times S)(k)$. We see that $\underline{\textup{DF}}^d_{k,\mathcal{E}_{\lambda}}(X_s)$ is mapped to $s$ under the composition
    \[
        \underline{\textup{DF}}^d_{k,\mathcal{E}_{\lambda}}(X_s) \rightarrow\underline{\textup{DF}}_{S,\mathcal{E}}(X) \rightarrow U\times S \rightarrow S.
    \]
    Since $\underline{\textup{DF}}^d_{k,\mathcal{E}_{\lambda}}(X_s)$ is non-empty, we find that $s\in\im(\varphi)$. This shows that indeed $\Sigma^d_{\textup{DF}}\subseteq \im(\varphi)$, and hence that $\varphi$ is dominant by our assumption that $\Sigma^d_{\textup{DF}}$ is dense.

    Since $\varphi$ is dominant and quasi-compact, the generic point $\eta$ of $S$ is in the image of $\varphi$ \sp{01RL} and thus the fiber $\underline{\textup{DF}}^d_{K,\mathcal{E}_K}(X_K)$ of $\varphi$ over $\eta$ is non-empty. In particular, it has a $\overline{K}$-point $x$. Let $\mu$ be the image of~$x$ inside $U(\overline{K})$, so that the fiber of $\underline{\textup{DF}}^d_{K,\mathcal{E}_K}(X_K) \rightarrow U_K$ over $\mu$, which is given by $\underline{\textup{DF}}^d_{\overline{K}, \mathcal{E}_\mu}$, is non-empty. Then $\underline{\textup{DF}}^d_{\overline{K}, \mathcal{E}_\mu}$ also has a $\overline{K}$-point. Such a point corresponds precisely to a finite morphism of degree $m$ from $X_{\overline{K}}$ to a DF curve of type $\frac{d}{m}$ for some $m$ dividing $d$.
\end{proof}

\section{Proof of the main theorem}\label{sec:main_result}
Finally, we prove \Cref{thm:main_result_intro}.

\begin{theorem}\label{thm:main_result}
    Let $K$ be a function field over a field $k$ of characteristic 0, let $X/K$ be a smooth projective curve and set $d = \min(\wp(X/K))$. Assume that one of the following holds.
    \begin{enumerate}
        \item[(i)] $k = \mathbb{Q}$; or,
        \item[(ii)] $k$ is algebraically closed in $K$ and $X_{\overline{K}}$ is not dominated by a cover of degree at most $d$ of a constant curve.
    \end{enumerate}
    Then the following statements hold.
    \begin{enumerate}
        \item If $d \le 3$, then $X_{\overline{K}}$ is a cover of degree $d$ of $\bbP^1$ or of an elliptic curve over $\overline{K}$.
        \item If $d = 4,5$, then either $X_{\overline{K}}$ is a cover of degree $d$ of $\bbP^1$ or of an elliptic curve over $\overline{K}$, or $X_{\overline{K}}$ is (the normalization of) a DF curve of type $d$.
    \end{enumerate}
\end{theorem}

\begin{remark}
    For $d=2$, the result follows directly from \Cref{thm:degree_d_points_from_abvar_intro} together with \cite[Proposition~1]{HS91}.
\end{remark}

\begin{proof}[Proof of \Cref{thm:main_result}]
    We may assume that $\textup{gon}(X_{\overline{K}}) > d$, because if $\textup{gon}(X_{\overline{K}}) = d$ we are done. By \Cref{thm:degree_d_points_from_abvar_intro}, there exists an abelian variety $A$ over $\overline{K}$ such that we have a diagram
    \begin{equation}\label{diag_pf_main_thm_1}
        \begin{tikzcd}
	Z & {X_{\overline{K}}} \\
	A.
	\arrow[two heads, from=1-1, to=1-2]
	\arrow["{\textup{deg }d}"', from=1-1, to=2-1]
\end{tikzcd}
    \end{equation}
    Then there exists an affine $\mathbb{Q}$-scheme $S$ of finite type and smooth over $\mathbb{Q}$, with function field~$K(S)$ inside~$\overline{K}$, over which this diagram is defined. More precisely, there exist models $\mathcal{X}$, $\mathcal{A}$ and $\mathcal{Z}$ over $S$ of $X$, $A$ and $Z$ respectively, such that $\mathcal{X}$ is smooth and projective over $S$, such that $\mathcal{A}$ is an abelian scheme over $S$ and such that we have a diagram
    \begin{equation}\label{diag_pf_main_thm_2}
        \begin{tikzcd}
	{\mathcal{Z}} & {\mathcal{X}} \\
	{\mathcal{A}}
	\arrow[two heads, from=1-1, to=1-2]
	\arrow["{\textup{deg }d}"', from=1-1, to=2-1]
\end{tikzcd}
    \end{equation}
    over $S$, satisfying that the base change along $\spec(\overline{K}) \rightarrow S$ is isomorphic to Diagram \ref{diag_pf_main_thm_1}. We are going to apply the results of \Cref{sec:specialization_lemmas} to $\mathcal{X}_{\ol{\mathbb{Q}}}\rightarrow S_{\ol{\mathbb{Q}}}$.

    For every closed point $s\in S$, the pullback of Diagram \ref{diag_pf_main_thm_2} along $\spec(\kappa(s))\rightarrow S$ shows that~$\mathcal{X}_s$ is dominated by a cover of degree $d$ of an abelian variety. In particular, $\mathcal{X}_s$ has a dense set of points of degree at most $d$, so that $\min(\wp(\mathcal{X}_{{s}}/\kappa(s))) \le d$.

    Note that $\kappa(s)$ is a number field, so we can apply \Cref{thm:geom_char_deg_d_nf}. We see that if $d\le 3$, we have that~$\mathcal{X}_{{s}, \overline{\mathbb{Q}}}$ is a cover of degree at most $d$ of $\mathbb{P}^1$ or of an elliptic curve over $\overline{\mathbb{Q}}$, and if $d=4,5$, we get the additional possibility that $\mathcal{X}_{s, \overline{\mathbb{Q}}}$ is a DF curve of type at most $d$, or the normalization of such a curve.

    First consider the case that $d\le 3$. By the above, there exists $d'\le d$, such that
    \begin{align*}
        \Sigma^{d'}_0 & := \bigl\{\overline{s}\in S_{\ol{\mathbb{Q}}}(\ol{\mathbb{Q}}) \bigm| X_{\overline{s}} \textup{ is a cover of degree }d'\textup{ of }\bbP^1 \bigr\}, \textup{ or}                \\
        \Sigma^{d'}_1 & := \bigl\{\overline{s}\in S_{\ol{\mathbb{Q}}}(\ol{\mathbb{Q}})  \bigm| X_{\ol{s}} \textup{ is a cover of degree }d'\textup{ of some elliptic curve over }\ol{\mathbb{Q}} \bigr\}
    \end{align*}
    is dense in $S_{\ol{\mathbb{Q}}}$. Then \Cref{lem:specialization_thm_genus_0} and \Cref{lem:specialization_thm_genus_1} applied to $\mathcal{X}_{\ol{\bbQ}}\rightarrow S_{\ol{\bbQ}}$ show that $\mathcal{X}_{\overline{K(S)}}$ is a cover of degree $d'$ of either $\mathbb{P}^1$ or of an elliptic curve. Consequently, the base change $\mathcal{X}_{\overline{K}} = X_{\overline{K}}$ is a cover of degree $d'$ of either $\mathbb{P}^1$ or of an elliptic curve as well. This implies $\min(\wp(X/K))\le d'$, so we must have $d' =d$. This proves part (1).

    Now assume that $d=4,5$. By the above, there exists $d'\le d$, such that $\Sigma^{d'}_0$,$\Sigma^{d'}_1$ or
    \begin{align*}
        \Sigma^{d'}_{\textup{DF}} & := \Biggl\{s\in  S_{\ol{\mathbb{Q}}}(\ol{\mathbb{Q}}) \Biggm|
        \vcenter{
            \setbox0=\hbox{there exists an $m$ dividing $d'$ such that $X_s$ is}
            \copy0
            \hbox to\wd0{\hfill a cover of degree $m$ of a DF curve of type $\frac{d'}{m}$\hfill}
        }
        \Biggr\}
    \end{align*}
    is dense in $S$. If either $\Sigma^{d'}_0$ or $\Sigma^{d'}_1$ is dense in $S$, the same arguments as in the $d\le 3$ case show that $X_{\overline{K}}$ is a cover of degree $d$ of $\mathbb{P}^1$ or of an elliptic curve over $\overline{K}$. Suppose that $\Sigma^{d'}_{\textup{DF}}$ is dense in $S$. Then by \Cref{lem:specialization_thm_DF}, there is an $m$ dividing $d'$ such that $\mathcal{X}_{\overline{K(S)}}$ admits a morphism of degree $m$ to a DF curve of type $\frac{d'}{m}$. Then the same is true for the base change $\mathcal{X}_{\overline{K}} = X_{\overline{K}}$. In particular, $X$ has a potentially dense set of points of degree $d'$ coming from the degree $\frac{d'}{m}$ points of this DF curve, so we must have $d' = d$. There are now three possibilities:
    \begin{itemize}
        \item $X_{\overline{K}}$ is a cover of degree $d$ of a DF curve of type 1. Then we are done, because a DF curve of type 1 is an elliptic curve \cite[4.1]{DF93}.
        \item $X_{\overline{K}}$ admits a morphism of degree 1 to a DF curve of type~$d$, but then $X_{\overline{K}}$ is a DF curve of type~$d$ itself, or the normalization of one.
        \item $d = 4$, and $X_{\overline{K}}$ is a cover of degree $2$ of a DF curve $C$ of type 2. Then a model of $C$ over some finite field extension of $K$ has a dense set of points of degree 2. Hence, by part (1) of the theorem, $C$ admits a morphism of degree 2 to either $\mathbb{P}^1$ or to an elliptic curve. Precomposing this morphism with $X_{\overline{K}} \rightarrow C$ shows that $X_{\overline{K}}$ is a cover of degree 4 of either $\mathbb{P}^1$ or of an elliptic curve.
    \end{itemize}
    This concludes the proof.
\end{proof}

\small

\bibliography{references.bib}

\newcommand{\etalchar}[1]{$^{#1}$}
\providecommand{\bysame}{\leavevmode\hbox to3em{\hrulefill}\thinspace}
\providecommand{\MR}{\relax\ifhmode\unskip\space\fi MR }
\providecommand{\MRhref}[2]{%
  \href{http://www.ams.org/mathscinet-getitem?mr=#1}{#2}
}
\providecommand{\href}[2]{#2}
\begin{thebibliography}{BEL{\etalchar{+}}19}

\bibitem[AH91]{AH91}
D.~Abramovich and J.~Harris, \emph{Abelian varieties and curves in {$W_d(C)$}}, Compositio Math. \textbf{78} (1991), no.~2, 227--238. \MR{1104789}

\bibitem[BEL{\etalchar{+}}19]{BELOV19}
A.~Bourdon, \"{O}. Ejder, Y.~Liu, F.~Odumodu, and B.~Viray, \emph{On the level of modular curves that give rise to isolated {$j$}-invariants}, Adv. Math. \textbf{357} (2019), 106824, 33. \MR{4016915}

\bibitem[Bui92]{Bui92}
A.~Buium, \emph{Intersections in jet spaces and a conjecture of {S}. {L}ang}, Ann. of Math. (2) \textbf{136} (1992), no.~3, 557--567. \MR{1189865}

\bibitem[Deb01]{Deb01}
O.~Debarre, \emph{Higher-dimensional algebraic geometry}, Universitext, Springer-Verlag, New York, 2001. \MR{1841091}

\bibitem[DF93]{DF93}
O.~Debarre and R.~Fahlaoui, \emph{Abelian varieties in {$W^r_d(C)$} and points of bounded degree on algebraic curves}, Compositio Math. \textbf{88} (1993), no.~3, 235--249. \MR{1241949}

\bibitem[Fal83]{Fal83}
G.~Faltings, \emph{Endlichkeitss\"{a}tze f\"{u}r abelsche {V}ariet\"{a}ten \"{u}ber {Z}ahlk\"{o}rpern}, Invent. Math. \textbf{73} (1983), no.~3, 349--366. \MR{718935}

\bibitem[Fal94]{Fal94}
\bysame, \emph{The general case of {S}. {L}ang's conjecture}, Barsotti {S}ymposium in {A}lgebraic {G}eometry ({A}bano {T}erme, 1991), Perspect. Math., vol.~15, Academic Press, San Diego, CA, 1994, pp.~175--182. \MR{1307396}

\bibitem[FJ74]{FJ74}
G.~Frey and M.~Jarden, \emph{Approximation theory and the rank of abelian varieties over large algebraic fields}, Proc. London Math. Soc. (3) \textbf{28} (1974), 112--128. \MR{337997}

\bibitem[Gao25]{Gao25}
G.~Gao, \emph{A complete proof of the geometric bombieri-lang conjecture for ramified covers of abelian varieties}, \href{https://arxiv.org/pdf/2511.17010}{arXiv:2511.17010}, 2025.

\bibitem[GW10]{GW10}
U.~G\"{o}rtz and T.~Wedhorn, \emph{Algebraic geometry {I}}, Advanced Lectures in Mathematics, Vieweg + Teubner, Wiesbaden, 2010. \MR{2675155}

\bibitem[Gra65]{Gra65}
H.~Grauert, \emph{Mordells {V}ermutung \"{u}ber rationale {P}unkte auf algebraischen {K}urven und {F}unktionenk\"{o}rper}, Inst. Hautes \'{E}tudes Sci. Publ. Math. (1965), no.~25, 131--149. \MR{222087}

\bibitem[Gro61]{Gro61}
A.~Grothendieck, \emph{Techniques de construction et th\'{e}or\`emes d'existence en g\'{e}om\'{e}trie alg\'{e}brique. {IV}. {L}es sch\'{e}mas de {H}ilbert}, S\'{e}minaire {B}ourbaki, {V}ol. 6, Soc. Math. France, Paris, 1961, pp.~Exp. No. 221, 249--276. \MR{1611822}

\bibitem[HS91]{HS91}
J.~Harris and J.~Silverman, \emph{Bielliptic curves and symmetric products}, Proc. Amer. Math. Soc. \textbf{112} (1991), no.~2, 347--356. \MR{1055774}

\bibitem[Har77]{Har77}
R.~Hartshorne, \emph{Algebraic geometry}, Graduate Texts in Mathematics, No. 52, Springer-Verlag, New York-Heidelberg, 1977. \MR{463157}

\bibitem[Hru96]{Hru96}
E.~Hrushovski, \emph{The {M}ordell-{L}ang conjecture for function fields}, J. Amer. Math. Soc. \textbf{9} (1996), no.~3, 667--690. \MR{1333294}

\bibitem[JK20]{JK20}
A.~Javanpeykar and L.~Kamenova, \emph{Demailly's notion of algebraic hyperbolicity: geometricity, boundedness, moduli of maps}, Math. Z. \textbf{296} (2020), no.~3-4, 1645--1672. \MR{4159843}

\bibitem[KV25]{KV25}
B.~Kadets and I.~Vogt, \emph{Subspace configurations and low degree points on curves}, Adv. Math. \textbf{460} (2025), Paper No. 110021, 36. \MR{4828751}

\bibitem[Kol96]{Kol96}
J.~Koll\'{a}r, \emph{Rational curves on algebraic varieties}, Ergebnisse der Mathematik und ihrer Grenzgebiete. 3. Folge. A Series of Modern Surveys in Mathematics, vol.~32, Springer-Verlag, Berlin, 1996. \MR{1440180}

\bibitem[Laz04]{Laz04I}
R.~Lazarsfeld, \emph{Positivity in algebraic geometry. {I}}, Ergebnisse der Mathematik und ihrer Grenzgebiete. 3. Folge. A Series of Modern Surveys in Mathematics, vol.~48, Springer-Verlag, Berlin, 2004. \MR{2095471}

\bibitem[Man63]{Man63}
J.~I. Manin, \emph{Proof of an analogue of {M}ordell's conjecture for algebraic curves over function fields}, Dokl. Akad. Nauk SSSR \textbf{152} (1963), 1061--1063. \MR{154868}

\bibitem[Mil86]{Mil86}
J.~S. Milne, \emph{Jacobian varieties}, Arithmetic geometry ({S}torrs, {C}onn., 1984), Springer, New York, 1986, pp.~167--212. \MR{861976}

\bibitem[Mil08]{Mil08}
\bysame, \emph{Abelian varieties (v2.00)}, 2008, Available at \url{www.jmilne.org/math/}, pp.~166+vi.

\bibitem[Mum70]{Mum70}
D.~Mumford, \emph{Abelian varieties}, Tata Institute of Fundamental Research Studies in Mathematics, vol.~5, Published for the Tata Institute of Fundamental Research, Bombay; by Oxford University Press, London, 1970. \MR{282985}

\bibitem[Nit05]{Nit05}
N.~Nitsure, \emph{Construction of {H}ilbert and {Q}uot schemes}, Fundamental algebraic geometry, Math. Surveys Monogr., vol. 123, Amer. Math. Soc., Providence, RI, 2005, pp.~105--137. \MR{2223407}

\bibitem[Poo17]{Poo17}
B.~Poonen, \emph{Rational points on varieties}, Graduate Studies in Mathematics, vol. 186, American Mathematical Society, Providence, RI, 2017. \MR{3729254}

\bibitem[Ray83]{Ray83}
M.~Raynaud, \emph{Around the {M}ordell conjecture for function fields and a conjecture of {S}erge {L}ang}, Algebraic geometry ({T}okyo/{K}yoto, 1982), Lecture Notes in Math., vol. 1016, Springer, Berlin, 1983, pp.~1--19. \MR{726419}

\bibitem[Sil09]{Sil09}
J.~H. Silverman, \emph{The arithmetic of elliptic curves}, second ed., Graduate Texts in Mathematics, vol. 106, Springer, Dordrecht, 2009. \MR{2514094}

\bibitem[Sp18]{Sp18}
{The Stacks Project Authors}, \emph{\textit{Stacks Project}}, \url{https://stacks.math.columbia.edu}, 2025.

\bibitem[VV24]{VV24}
B.~Viray and I.~Vogt, \emph{Isolated and parameterized points on curves}, \href{https://arxiv.org/pdf/2406.14353}{arXiv:2406.14353}, 2024.

\bibitem[XY23a]{XY23b}
J.~Xie and X.~Yuan, \emph{The geometric bombieri-lang conjecture for ramified covers of abelian varieties}, \href{https://arxiv.org/pdf/2308.08117}{arXiv:2308.08117}, 2023.

\bibitem[XY23b]{XY23a}
\bysame, \emph{Partial heights, entire curves, and the geometric bombieri-lang conjecture}, \href{https://arxiv.org/pdf/2305.14789}{arXiv:2305.14789}, 2023.

\end{thebibliography}

\bibliographystyle{amsalpha}

\end{document}